\documentclass[10pt]{article}
\usepackage{amsfonts}
\usepackage{amsthm}
\usepackage{amsmath}
\usepackage{amssymb}
\usepackage{tikz,tikz-cd}
\usepackage{mathabx}
\usetikzlibrary{matrix,arrows,decorations.pathmorphing}
\usepackage{cite}
\usepackage{hyperref}
\usepackage{cleveref}
\usepackage{tabularx}
\usepackage{longtable}

\newcommand{\ds}{\displaystyle}
\newcommand{\normal}{\trianglelefteq}
\newcommand{\F}{\mathbb{F}}
\newcommand{\Z}{\mathbb{Z}}
\newcommand{\Q}{\mathbb{Q}}

\newcommand{\Gal}{\textnormal{Gal}}
\newcommand{\Hom}{\textnormal{Hom}}
\newcommand{\Cl}{\textnormal{Cl}}
\newcommand{\im}{\textnormal{im }}

\newcommand{\inv}{\rm inv}

\newcommand{\Aut}{\textnormal{Aut}}
\newcommand{\Surj}{\textnormal{Surj}}

\newcommand{\disc}{\textnormal{disc}}

\newcommand{\pK}{\mathfrak{p}}
\newcommand{\ind}{\textnormal{ind}}

\usepackage[margin=1.5in]{geometry}

\newtheorem*{lemma*}{Lemma}

\newtheorem{theorem}{Theorem}[section]
\newtheorem{lemma}[theorem]{Lemma}
\newtheorem{definition}[theorem]{Definition}
\newtheorem{corollary}[theorem]{Corollary}

\newtheorem{conjecture}[theorem]{Conjecture}

\title{The Weak Form of Malle's Conjecture and Solvable Groups}
\author{Brandon Alberts}

\begin{document}
\maketitle

\abstract{For a fixed finite solvable group $G$ and number field $K$, we prove an upper bound for the number of $G$-extensions $L/K$ with restricted local behavior (at infinitely many places) and $\inv(L/K)<X$ for a general invariant $``\inv"$. When the invariant is given by the discriminant for a transitive embedding of a nilpotent group $G\subset S_n$, this realizes the upper bound given in the weak form of Malle's conjecture. For other solvable groups, the upper bound depends on the size of torsion of the class group of number fields with fixed degree. In particular, the bounds we prove realize the upper bound given in the weak form of Malle's conjecture for the transitive embedding of a solvable group $G\subset S_n$ if we assume that for each finite abelian group $A$ the average size of class group torsion $|\Hom(\Cl(L),A)|$ is smaller than $X^{\epsilon}$ as $L/K$ varies over certain families of extensions with $\inv(L/K)<X$.}

\section{Introduction}

Let $K$ be a number field. One of the biggest questions in arithmetic statistics is counting number fields, and Malle specifically studied counting degree $n$ extensions $L/K$ inside of a fixed algebraic closure $\overline{K}$ when ordered by relative discriminant $\disc(L/K)$ in \cite{malle2002} and \cite{malle2004}. The Galois group of $L/K$ (or rather, of the Galois closure of $L/K$) is $\Gal(L/K)\subset S_n$ a transitive subgroup acting on the $n$ embeddings $L\hookrightarrow \overline{K}$. Let $G\subset S_n$ be a transitive subgroup and
\[
N(K,G;X) = \#\{L/K \mid \Gal(L/K)=G, {\rm Nm}_{K/\Q}(\disc(L/K))<X\}\,.
\]
Malle conjectured that there exists a positive constant $c(K,G)$ such that
\[
N(K,G;X) \sim c(K,G)X^{1/a(G)}(\log X)^{b(K,G)-1}
\]
asymptotically as $X\to \infty$, where $a(G)$ and $b(K,G)$ are explicit positive constants. Here we take $f(X) \sim g(X)$ to mean $\lim_{X\to \infty} f(X)/g(X) = 1$. This is often referred to as the strong form of Malle's conjecture. The values $a(G)$ and $b(K,G)$ have been verified in the following cases:
\begin{itemize}
\item{$G$ an abelian group by Wright \cite{wright1989},}

\item{$G=S_n$ for $n=3$ with $K=\Q$ by Davenport-Heilbronn \cite{davenport-heilbronn1971}, $n=3$ with $K$ arbitrary by Datskovsky-Wright \cite{datskovsky-wright1988}, $n=4,5$ with $K=\Q$ by \cite{bhargava2005}, and $n=4,5$ with $K$ arbitrary by Bhargava-Shankar-Wang \cite{bhargava-shankar-wang2015},}

\item{$S_3\subset S_6$ by Bhargava-Wood \cite{bhargava-wood2007},}

\item{$D_4\subset S_4$ by Cohen-Diaz y Diaz-Olivier \cite{cohen-diaz-y-diaz-olivier2002},}

\item{$Q_{4m}\subset S_{4m}$ the generalized quaternion group of order $4m$ by Kl\"uners \cite{klunersHab2005},}

\item{$S_n\times A\subset S_{n\#A}$ for $A\subset S_{|A|}$ an abelian group in its regular representation with $(|A|,2)=1$ for $n=3$ and $(|A|,n!)=1$ for $n=4,5$ by Wang \cite{jwang2017},}

\item{$C_2\wr H$ for certain groups $H$ by Kl\"uners \cite{kluners2012}.}
\end{itemize}
The author is also aware of upcoming results for $D_4\subset S_8$ \cite{shankar-varma2019} and a large family of imprimitive groups (including many wreath products of the form $T\wr B$ for $T$ abelian or $S_3$) \cite{lemke-oliver-jwang-wood2019}. Similar results are known when the extensions $L/K$ are ordered by other invariants, for example Wood \cite{wood2009} proves the analogous results for abelian groups ordered by conductor. The constant $c(K,G)$, while known to be positive in all of the above cases, is known explicitly in far fewer cases, such as for abelian extensions over $\Q$ by M\"aki \cite{maki1985} and cyclic quartic extensions over general $K$ by Cohen-Diaz y Diaz-Olivier \cite{cohen-diaz-y-diaz-olivier2005}.

Unfortunately, the conjecture is not true in this form, as Kl\"uners \cite{kluners2005} provided a counter-example for $G=C_3\wr C_2\subset S_6$ for which $b(K,G)$ is too small. There have been proposed corrections for $b(K,G)$ by T\"urkelli \cite{turkelli2015}, but the $1/a(G)$ exponent is still widely believed to be correct. This leads to the weak form of Malle's conjecture, which will be the major focus of study of this paper:
\begin{conjecture}[weak form of Malle's conjecture]
For any $\epsilon >0$,
\[
X^{1/a(G)}\ll N(K,G;X) \ll X^{1/a(G)+\epsilon}\,,
\]
where we define $\ind:S_n\rightarrow \Z$ by $\ind(\sigma) = n - \#\{\text{orbits of }\sigma\}$ and
\[
a(G) = \min_{g\in G-\{1\}} \ind(\sigma)\,.
\]
\end{conjecture}
Here we write $f(X)\gg g(X)$ if $\infty \ge \limsup_{X\rightarrow\infty}f(X)/g(X)>0$. This has been proven in more cases, notably:
\begin{itemize}
\item{$G$ any nilpotent group in its regular representation $G\subset S_{|G|}$ by Kl\"uners-Malle \cite{kluners-malle2004},}

\item{Kl\"uners-Malle also proved the predicted upper bound for $\ell$-groups in any representation \cite{kluners-malle2004},}

\item{For $D_{p}$ the dihedral group of order $2p$ for a prime $p$ (in both the degree $p$ and $2p$ representations) Kl\"uners proved the lower bound unconditionally and the upper bound conditional on the Cohen-Lenstra heuristics for the average size of torsion in class groups of quadratic fields \cite{kluners2006}.}

\item{In a separate project announced around the same time as this paper, Mehta proves Malle's predicted upper bounds for Frobenius groups $G\cong F\rtimes H\subset S_n$ with abelian kernel with $n=|F|$ or $|G|$, conditional on Malle's predicted upper bound for $H\subset S_{|H|}$ and the $\ell$-torsion conjecture \cite{mehta2019},}

\item{$A_4, C_3^2 \rtimes C_4 \subset S_6$ announced around the same time as this paper by Mehta \cite{mehta2019}.}
\end{itemize}
There are no known counterexamples to the weak form of Malle's conjecture.

For other groups, not even the upper bound of $X^{1/a(G)+\epsilon}$ is known. The subject of studying just this upper bound is a vibrant area itself, and centers on a folklore conjecture attributed to Linnik: Let $N_{K,n}(X)$ be the number of extensions $L/K$ of degree $n$ with ${\rm Nm}_{K/\Q}(\disc(L/K))<X$. Then Linnik's conjecture predicts that
\[
N_{K,n}(X) \sim C_{K,n} X
\]
as $X\to \infty$. Linnik's conjecture would follow from the strong form of Malle's conjecture, while the weak form of Malle's conjecture would imply $\log N_{K,n}(X) \sim \log X$. Progress towards this conjecture has been slow; the best general bounds are due to Schmidt \cite{schmidt1995}, which state
\[
N_{K,n}(X) \ll X^{\frac{n+2}{4}}\,.
\]
This was improved upon by Ellenberg-Venkatesh \cite{ellenberg-venkatesh2006} for large $n$, who proved that there exist constants $A_n$ depending on $n$ and an absolute constant $C$ such that
\[
N_{K,n}(X) \ll \left(X |\disc(K/\Q)|^n A_n^{[K:\Q]}\right)^{\exp(C\sqrt{\log n})}\,,
\]
which shows in particular that
\[
\limsup_{X\rightarrow \infty} \frac{\log N_{K,n}(X)}{\log X} \ll_n n^{\epsilon}\,.
\]
Keeping in step with the philosophy of Malle's conjecture, Dummit \cite{dummit2018} proved an upper bound which improves upon Schmidt's bounds in many cases when $G\subset S_n$ is a \emph{proper} transitive subgroup. If any subgroup $G'\le G$ properly containing a point stabilizer has index at most $t$ in $G$, then Dummit shows
\[
N(K,G;X) \ll X^{\frac{1}{2(n-t)}\left(\sum_{i=1}^{n-1}\deg(f_{i+1}) - \frac{1}{[K:\Q]}\right)+\epsilon}
\]
for $f_1,...,f_n$ a set of primary invariants of $G$ for which $\deg f_i\le i$. This result significantly improves upon the bounds given by Schmidt in many cases, but is still not very close to the bound predicted by Malle. In particular, this exponent is larger than $\frac{1}{2} - \frac{1}{2n[K:\Q]}$ which is very close to $1/2$ for large $K$, although many groups have $1/a(G)<1/2$ (namely all groups in the regular representation other than $G=C_2$).

When one asks the analogous question over function fields $\F_q[T]$, a preprint of Ellenberg-Tran-Westerland \cite{ellenberg-tran-westerland2017} shows that when $q\gg |G|$ Malle's predicted upper bound is satisfied. This gives strong evidence that Malle's predicted upper bounds should hold over number fields, although the methods used for function fields do not appear to transfer to number fields.

In this paper, we prove upper bounds for number fields inductively by relating the number of $G$-extensions to the number of $N$-extensions and $G/N$ extensions for a normal subgroup $N\normal G$. The main result is as follows:

\begin{theorem}\label{thm:mainintro}
Let $K$ be a number field, $G\subset S_n$ a transitive subgroup, and $N\normal G$ an abelian normal subgroup. Then, for any $\epsilon>0$, $N(K,G;X)$ is
\[
\ll \max\left\{X^{1/a(N)+\epsilon} , \sum_{\substack{\Gal(L/K)\cong G/N\\ {\rm Nm}_{K/\Q}(\overline{\disc}(L/K))\le X}}X^{\epsilon}\Big|\Hom\left(\Cl\left(L^{C_G(N)/N}\right),N\right)\Big|\right\}\,,
\]
where the sum is over $G/N$-extensions $L/K$ ordered by a certain invariant $\overline{\disc}(L/K) \le X$, $C_G(N)$ denotes the centralizer of $N$ in $G$, and
\[
a(N) = \min_{g\in N-\{1\}} \ind(g)\,.
\]
\end{theorem}

The explicit definition of $\overline{\disc}$ can be found in Definition \ref{def:inv}. This result allows one to pass from bounds on the average size of class group torsion over certain subfields of $G/N$-extensions to bounds on $N(K,G;X)$. The goal is to use this result inductively in order to lift bounds on the number of $G/N$-extensions to bounds on the number of $G$-extensions.

The idea behind the proof is similar to that of other inductive results in number field counting, which are used to prove the strong form for $C_2\wr H$ \cite{kluners2012}, $A\times S_n$ \cite{jwang2017}, and for a large family of nonabelian groups in an upcoming preprint \cite{lemke-oliver-jwang-wood2019}, as well as upper bounds for $D_p$ \cite{kluners2006} and Frobenius groups with abelian kernel \cite{mehta2019}. In all previous work on the topic, this method has been employed successfully specifically when the normal subgroup $N\normal G\subset S_n$ fixes a partition of the set $\{1,2,...,n\}$ under the transitive action. This is exactly the situation in which the discriminant has a product formula
\[
D_{F/K} = {\rm Nm}_{L/K}(D_{F/L})\cdot D_{L/K}^{[F:L]}\,.
\]
(This is also called a Brauer relation in \cite{mehta2019}.) This relation is heavily used to pass from counting results on $G/N$-extension to counting results on $G$-extensions by relating the two discriminants.

The main improvement demonstrated by Theorem \ref{thm:mainintro} compared to previous inductive approaches is that we are allowed to take $N$ to be any abelian normal subgroup, without needing to assume it preserves a partition of the $\{1,...,n\}$. We do this by directly relating $G$-extensions ordered by discriminant to $G/N$-extensions ordered by $\overline{\disc}$, an invariant which need not agree with a discriminant, allowing us to bypass the need for a discriminant product formula entirely. This removes many of the restrictions of past work on the subject, and by proving a generalization of Theorem \ref{thm:mainintro} under any sufficiently general ordering we are still able to use this result inductively to produce upper bounds (see Theorem \ref{thm:main}). The freedom of being able to take any abelian normal subgroup allows us to prove very general results, in particular for solvable groups by inducting along a normal series with abelian factors. We group these according to results which are conditional on other conjectures and results which are unconditional.

\subsection{Conditional Results}

Theorem \ref{thm:mainintro} gives a concrete relation between bounds for $N(K,G;X)$ and the size of torsion in class groups. It is widely believed that torsion in class groups is ``small'' on average. We can phrase this a little more explicitly as follows:

\begin{conjecture}[Average Torsion Conjecture]
Let $K$ be a number field, $G$ a finite group, $A$ a finite abelian group, $\mathcal{F}$ a ``nice'' family of $G$-extensions $L/K$, and $\inv$ an admissible ordering ( such as the discriminant). Then the average size of $|\Hom(\Cl(L),A)|$ over extensions $L/K\in \mathcal{F}$ with ${\rm Nm}_{K/\Q}(\inv(L/K)) \le X$ is $\ll X^{\epsilon}$.
\end{conjecture}

The author is not aware of a place in the literature where this conjecture appears in such great generality, but generally speaking the behavior is not expected to change when $\inv$ is some admissible ordering other than a discriminant (we again refer to Definition \ref{def:inv} for the definition of an admissible ordering). In many cases, it is expected that the average value is actually finite, see \cite{cohen-lenstra1984,bartel-lenstra2017} for example. Davenport-Heilbronn's \cite{davenport-heilbronn1971} work on cubic extensions implies that the average value of $|\Hom(\Cl(L),C_3)|$ as $L/\Q$ varies over quadratic extensions is finite, while Gauss's genus theory implies the average value of $|\Hom(\Cl(L),C_2)|$ as $L/\Q$ varies over quadratic extensions grows like $\log X$.

This average value of class group torsion appears in Theorem \ref{thm:mainintro}. By inducting along a composition series, we are able to prove the following:

\begin{corollary}\label{cor:solvableintro}
Let $K$ be a number field and $G\subset S_n$ a solvable transitive subgroup. If the Average Torsion Conjecture is true, then
\[
N(K,G;X) \ll X^{1/a(G)+\epsilon}\,.
\]
\end{corollary}

This verifies Malle's predicted upper bounds for all solvable groups, conditional only on bounds for the average growth of torsion in class groups, and motivates the title of the paper. We need to assume such bounds under various admissible orderings, not just the discriminant, in order to prove such a general result. If we assume a stronger conjecture, we can restrict ourselves to only considering the discriminant ordering:

\begin{conjecture}[$\ell$-Torsion Conjecture]
Let $K$ be a number field, $n$ a positive integer, and $\ell$ a prime number. Then $|\Cl(L)[\ell]|\ll |\disc(L/\Q)|^{\epsilon}$ as $L/K$ varies over all degree $n$ extensions.
\end{conjecture}

The $\ell$-torsion conjecture is a much stronger assumption, and as a consequence we are much farther from proving such a result. Although the $\ell$-torsion conjecture is widely believed, the only known cases of this conjecture come from genus theory, where we can show $|\Cl(L)[2]| = 2^{\omega(\disc(L/\Q))} \ll |\disc(L/\Q)|^{\epsilon}$ as $L/\Q$ varies over quadratic extensions. See \cite{brumer-silverman1996,duke1998,ellenberg-pierce-wood2017,zhang2005} for discussions on the $\ell$-torsion conjecture and known results.

The benefit of the $\ell$-torsion conjecture is that knowing this conjecture for a single ordering, such as the discriminant, implies that the conjecture is true over any admissible ordering. Thus the $\ell$-torsion conjecture implies the average torsion conjecture. This allows us to state a result which, although weaker than Corollary \ref{cor:solvableintro}, does not require any non-discriminant orderings to state:

\begin{corollary}\label{cor:solvableell}
Let $K$ be a number field and $G\subset S_n$ a solvable transitive subgroup. If the $\ell$-torsion conjecture is true, then
\[
N(K,G;X) \ll X^{1/a(G)+\epsilon}\,.
\]
\end{corollary}

\subsection{Unconditional Results}

We know that torsion in the class group can not be of an arbitrarily large size, as Minkowski's bounds imply a trivial bound for the size of the class group among $L/K$ of bounded degree, namely $|\Cl(L)|\ll |\disc(L/\Q)^{1/2}|$. This implies $|\Hom(\Cl(L),A)|\ll |\disc(L/\Q)|^{d/2}$ where $d$ is the number of generators of the finite abelian group $A$. Combining these bounds with Theorem \ref{thm:mainintro} can inductively produce upper bounds for $G$-extensions whenever $N$ is an abelian normal subgroup and we know bounds on the number of $G/N$-extensions. In certain cases, these inductive bounds will be better than the best known bounds for the group $G$ by other methods (such as \cite{schmidt1995,dummit2018}). We generally expect this to happen when inertia groups $I_{\pK}(L/K)\le G$ which are ``closer to the center'' of $G$ produce smaller exponents for $\pK$ in the discriminant of $L/K$. We include some data on what these bounds look like for groups of small degree in the appendix.

As an example, when $G=D_p\subset S_p$ is the dihedral group for $p$ an odd prime, then $a(G) = \frac{2}{p-1}$. Theorem \ref{thm:mainintro} reproduces the unconditional bounds proven by Kl\"uners \cite{kluners2006}:
\[
N(K,D_p;X)\ll X^{\frac{3}{p-1}+\epsilon}\,.
\]
These are significantly smaller that Schmidt's bound of $\frac{p+2}{4}$, and Dummit's bound which is at least $\frac{1}{2} - \frac{1}{2p[K:\Q]} \sim \frac{1}{2}$. In many ways, Theorem \ref{thm:mainintro} can be viewed as a vast generalization of Kl\"uners's results.

A special case of this result occurs when $N\normal G$ is a central subgroup, so that $\Cl(L^{C_G(N)/N})=\Cl(K)$ is independent of the average being taken. We then use an inductive argument in this setting combined with Minkowski's trivial bounds on the order of the class group in order to prove the following result:

\begin{corollary}\label{cor:hypercenter}
Let $K$ be a number field and $G\subset S_n$ be a transitive subgroup. Suppose there exists a normal subgroup $N\normal G$ such that
\begin{enumerate}
\item[(a)]{
$N/(N\cap Z^{\infty}(G))$ is abelian with $d$ generators, where $Z^{\infty}(G)$ denotes the hypercenter of $G$ i.e. the last term of the lower central series,
}

\item[(b)]{
There exists a positive real number $M$ such that for each $G/N$-extension $L/K$ ${\rm Nm}_{K/\Q}(\disc(L/K))\ll {\rm Nm}_{K/\Q}(\overline{\disc}(L/K))^M$ and
\[
X^{\frac{1}{2} d M} \cdot \#\{ L/K \mid \Gal(L/K)\cong G/N, \overline{\disc}(L/K)<X\} \ll X^{1/a(G)+\epsilon}\,.
\]
}
\end{enumerate}
Then
\[
N(K,G;X) \ll X^{1/a(G) + \epsilon}\,.
\]
\end{corollary}

This result unconditionally verifies Malle's predicted upper bounds for a large family of nonabelian groups. We remark that condition (b) is reminiscent of the conditions on the order of $A$ in Wang's proof of the strong form for $A\times S_n$ \cite{jwang2017}, and is of the same flavor as results in an upcoming preprint by Lemke Oliver-Wang-Wood \cite{lemke-oliver-jwang-wood2019} proving the strong form of Malle's conjecture for many more imprimitive groups. In fact, a condition of this nature appears in Malle's original work \cite[Proposition 5.2]{malle2002} used to provide evidence for the consistency of Malle's conjecture under wreath products. The main improvement in Corollary \ref{cor:hypercenter} compared to previous results comes from showing that the entire class group $\Cl(L)$ does not always contribute to upper bounds, only the subgroup $\Cl(L^{C_G(N)/N})$ does. When $N$ commutes with more elements, this reduces the effect the class group has on upper bounds.

If $G$ is itself a nilpotent group, we can take $N=G$ so that $N/(N\cap Z^{\infty}(G)) = 1$. In this case conditions (a) and (b) are trivially satisfied with $d=0$ and we prove the following corollary:

\begin{corollary}\label{cor:nilpotent}
Let $K$ be a number field and $G\subset S_n$ a nilpotent transitive subgroup. Then
\[
N(K,G;X) \ll X^{1/a(G)+\epsilon}\,.
\]
\end{corollary}

This is an improvement on the results of Kl\"uners-Malle \cite{kluners-malle2004}, who produce the predicted upper bound when $G$ is an $\ell$-group or $G$ is nilpotent in the regular representation.

\subsection{Admissible Orderings and Restricted Local Conditions}

The techniques of this paper work in great generality. As stated after Theorem \ref{thm:mainintro}, we will prove all of our results under any admissible ordering. In addition to considering the general problem being a key step in allowing us to perform inductive arguments, there has been recent interest in studying more general orderings such as the conductor or the product of ramified primes. In certain cases, these orderings have nicer properties than the discriminant and can be easier to work with. Wood \cite{wood2009} counts abelian extensions ordered by conductor, and shows some ways in which this invariant is nicer than the discriminant. Bartel-Lenstra \cite{bartel-lenstra2017}, Dummit \cite{dummit2018}, and Johnson \cite{johnson2017} continue this philosophy by studying different questions when ordering number fields by various invariants. We have made an effort to cater to this perspective, where the admissible orderings we consider are general enough to include other orderings previously considered in the literature.

We also prove analogous results for number fields with restricted local behavior at any number of places. Such bounds are considered in Bhargava-Shankar-Wang \cite{bhargava-shankar-wang2015}, Kl\"uners-Malle \cite{kluners-malle2004}, Wright \cite{wright1989}, and Wood \cite{wood2009} in cases with certain restricted local behaviors at finitely many places, and do not behave significantly differently (with the exception of those cases that fall under the Grunwald-Wang Theorem). Restricted local behavior at infinitely many places has also been considered, although the order of magnitude of $N(K,G;X)$ is not always expected to be the same under infinitely many restrictions on local behavior. Davenport-Heilbronn \cite{davenport-heilbronn1971} considers $S_3$ extensions with squarefree discriminant, and Bhargava \cite{bhargava2014} does the same for $S_4$ and $S_5$. Wright's proof of the strong form of Malle's conjecture for abelian extensions carries through for restricting ramification at infinitely many places as well \cite{wright1989}. The proof techniques we use require little to no modification in order to consider families of number fields with restricted local behaviors.

\subsection{Layout of the Paper}

In Section \ref{sec:definitions} we will give the explicit definitions and details needed to state the generalization of Theorem \ref{thm:mainintro} to an arbitrary admissible ordering under restricted local conditions (found in Theorem \ref{thm:main}). This will include the definition of an admissible ordering, and specifically the ordering $\overline{\disc}(L/K)$ appearing in Theorem \ref{thm:mainintro}.

We prove some technical lemmas in Section \ref{sec:technical}, then prove Theorem \ref{thm:main} in Section \ref{sec:proofmain}. The result is proven by breaking apart $\Hom(G_K,G)$ into the fibers of the push-forward $q_N:\Hom(G_K,G) \rightarrow \Hom(G_K,G/N)$, and proving upper bounds on the sizes of these fibers. We follow up with the proofs of the various conditional and unconditional corollaries in Section \ref{sec:cor}.

We conclude this paper with a discussion of possible improvements of this result in Section \ref{sec:remarks} relating to further study of nonsolvable groups or improving the unconditional bounds for solvable groups.

At the end of the paper we include an Appendix containing data on the unconditional bounds for solvable transitive subgroups of small degree.

\section{Main Definitions and Statements of Main Result}\label{sec:definitions}

Let $K$ be a number field, $G_K=\Gal(\overline{K}/K)$ the absolute Galois group, $P_K$ the set of places of $K$, and $I_K$ the set of ideals of the ring of integers in $K$. For each $\pK\in P$ let $D_\pK = \Gal(\overline{K}_\pK/K_\pK)$ be the absolute decomposition group and $I_\pK\normal D_\pK$ the absolute inertia group. These are subgroups of $G_K$ up to conjugacy, but throughout the paper we fix an embedding $D_\pK\hookrightarrow G_K$ for each $\pK\in P_K$.

We will prove the results of this paper for number fields in any sufficiently nice ordering. We take some cues from orders we want to consider, such as the discriminant or conductor, and define admissible orderings to only depend on the ramification data:
\begin{definition}\label{def:inv}
We define an \textbf{admissible ordering} (or \textbf{admissible invariant}) $\inv:\prod_{\pK}\Hom(D_\pK,G) \rightarrow I_K$ as follows:
\begin{enumerate}
\item[(a)]{there is a family of functions $\inv_\pK:\Hom(I_\pK,G)\rightarrow \Z_{\ge 0}$ for each place $\pK\in P_K$ such that
\[
\inv(\psi) = \prod_{\pK\in P} \pK^{\inv_\pK(\psi|_{I_\pK})}\,,
\]
}

\item[(b)]{
For all but finitely many places $\pK\in P_K$, $\gamma(I_\pK)=1$ if and only if $\inv_\pK(\gamma)=0$.
}
\end{enumerate}
We define $\inv:\Hom(G_K,G)\rightarrow I_K$ by $\inv(\pi) = \inv((\pi|_{D_\pK})_{\pK\in P})$.
\end{definition}

Then for any finite group $G$ and an admissible ordering $\inv$ we define the counting function
\[
N_{\inv}(K,G;X) = \#\{\pi:G_K\twoheadrightarrow G \mid {\rm Nm}_{K/\Q}(\inv(\pi))<X\}\,.
\]
If we choose a transitive subgroup $G\le S_n$, then the ordering given by
\[
\inv(\pi) = \disc(\text{degree } n \text{ \'etale algebra corresponding to }\pi)
\]
is admissible, and $N_{\inv}(K,G;X)=N(K,G;X)$ is the usual counting function studied by Malle, as described in the introduction. An alternate, but equivalent, expression for this invariant when $\pi$ is surjective is
\[
\inv(\pi) = \disc(L_\pi^{{\rm Stab}_G(1)})\,,
\]
where $L_\pi/K$ is the $G$-extension fixed by $\ker(\pi)$, and ${\rm Stab}_G(1)$ is the stabilizer of $1$ in $G\le S_n$.

We will give upper bounds on the asymptotic growth of this counting function by inductively comparing $N_{\inv}(K,G;X)$ to functions on $N$- and $G/N$-extensions for a normal subgroup $N\normal G$. We will need the following lemma for producing an admissible ordering $\overline\inv$ on $G/N$-extensions from an admissible ordering $\inv$ on $G$-extensions:

\begin{lemma}\label{lem:barinv}
Fix a number field $K$ and a finite group $G$. Let $\inv$ be an admissible ordering on $\Hom(G_K,G)$. Fix a normal subgroup $N\normal G$ with quotient map $q_N:G\rightarrow G/N$ and push-forward map $(q_N)_*:\Hom(-,G)\rightarrow \Hom(-,G/N)$. Define the family of functions $\overline\inv_\pK:\Hom(I_\pK,G/N)\rightarrow \Z_{\ge 0}$ defined by
\begin{align*}
\overline\inv_\pK(\gamma) = \min\{ \inv_\pK(\delta) : \delta\in (q_N)_*^{-1}(\gamma)\}\,.
\end{align*}
Then the function $\overline\inv:\prod_{\pK}\Hom(D_\pK,G/N)\rightarrow I_K$ defined by
\[
\overline\inv(\psi) = \prod_{\pK\in P_K} \pK^{\overline\inv_\pK(\psi)}
\]
is an admissible ordering.
\end{lemma}

\begin{proof}
Definition \ref{def:inv}(a) is trivially satisfied by $\overline\inv$.

For Definition \ref{def:inv}(b) suppose $\gamma(I_\pK)=1$. This implies that $\gamma$ is completely determined by the image of Frobenius $\gamma({\rm Fr}_\pK)$. Therefore there exists a lift $\delta:D_\pK\rightarrow G$ defined by $\delta(I_\pK)=1$ and $\delta({\rm Fr}_p)\in G$ some lift of $\gamma({\rm Fr}_p)\in G/N$. The admissibility of $\inv$ and $\delta(I_\pK)=1$ implies that $\inv_\pK(\delta)=0$ (except for finitely many $\pK$ providing exceptions), so by definition $0\le \overline\inv_\pK(\gamma)\le \inv_\pK(\delta)=0$.

Conversely, if $\overline\inv(\gamma)=0$ then there exists a lift $\delta:D_\pK\rightarrow G$ such that $\inv(\delta)=0$ by definition. Admissibility of $\inv$ then implies $\delta(I_\pK)=1$ (except for finitely many $\pK$ providing exceptions), which then implies $\gamma(I_\pK) = \delta(I_\pK)N/N = 1$.
\end{proof}

This is how we define $\overline{\disc}$ in Theorem \ref{thm:mainintro}. An alternate, but equivalent, definition can be given by considering the discriminant on local $G/N$-\'etale algebras:
\[
\overline{\disc}(L_\pK/K_\pK) = {\rm gcd}\left(\disc(F_\pK/K_\pK) \mid F_\pK^N = L_\pK^{\otimes |N|}\right)\,.
\]

We will additionally consider restricted local behavior: if we fix a family $\Sigma = (\Sigma_\pK)$ for $\Sigma_\pK\subset \Hom(D_\pK,G)$, we then define
\[
N_{\inv}(K,\Sigma;X) = \#\{\pi:G_K\twoheadrightarrow G \mid (\pi|_{D_\pK})\in\Sigma, {\rm Nm}_{K/\Q}(\inv(\pi))<X\}\,.
\]
We will give upper bounds for the growth of these functions as $X\rightarrow \infty$. We define the following notations:
\begin{definition}
Let $\Sigma=(\Sigma_\pK)$ be a family of local conditions $\Sigma_\pK\subset \Hom(D_\pK,G)$. We define the following:
\begin{enumerate}
\item[(a)]{
$\Surj^{\Sigma}(G_K,G) = \{\pi:G_K\twoheadrightarrow G: (\pi|_{D_\pK})\in \Sigma\}$,
}

\item[(b)]{
for any normal subgroup $N\normal G$ define $\Sigma(N) = \Sigma \cap \prod_{\pK} \Hom(D_\pK,N)$,
}

\item[(c)]{
for any normal subgroup $N\normal G$ with quotient map $q_N:G\rightarrow G/N$ and push-forward $(q_N)_*:\Hom(-,G)\rightarrow \Hom(-,G/N)$ we define $(q_n)_*\Sigma$ to be the family $((q_N)_*(\Sigma_{\pK}))$ with $(q_N)_*(\Sigma_\pK) \subset \Hom(D_\pK,G/N)$.
}
\end{enumerate}
\end{definition}

The weak form of Malle's conjecture defines the $a(G)$ invariant to be the smallest exponent of $\pK$ that can appear in the discriminant of a tamely ramified extension for all but finitely many places $\pK$, so we can make the analogous definition:
\begin{definition}\label{def:a}
Let $G$ be a finite group, $w$ admissible, and $\Sigma=(\Sigma_\pK)$. Define
\[
a_{\inv}(\Sigma) = \liminf_{\mathcal{N}_{K/\Q}(\pK)\rightarrow \infty}\min_{\substack{\gamma\in\Sigma_\pK\\ \gamma(I_\pK)\ne 1} } \inv_{\pK}(\gamma)\,,
\]
where we take the convention
\[
\min_{n\in S} n = \infty
\]
whenever the set $S$ is empty.

If $\Sigma_\pK=\Hom(D_\pK,G)$ is trivial for all places $\pK\in P_K$, we denote this by $a_{\inv}(G)$.
\end{definition}

If $G\subset S_n$ is a transitive subgroup, $\inv=\disc$ is the corresponding discriminant, and $\Sigma_\pK = \Hom(D_\pK,G)$ for all $\pK\in P_K$ then $a_{\inv}(\Sigma)=a(G)$ agrees with the invariant predicted by Malle. This remains true if we restrict local conditions at \emph{finitely many} places, but if we allow $\Sigma_{\pK} \ne \Hom(D_\pK,G)$ for all places $\pK\in P_K$ it is possible that $a_{\inv}(\Sigma)>a_{\inv}(G)$. This agree with the cases that are already known (for examples with restricted local conditions, see \cite{bhargava2014, davenport-heilbronn1971,wright1989}). We briefly describe a small example of this phenomenon for which $a_{\inv}(\Sigma)\ne a_{\inv}(G)$ is known to be the correct invariant:

\textbf{Example:} Let $G=C_4$, $\disc$ be the usual discriminant of $C_4$-extensions of $\Q$, and define the local conditions
\[
\Sigma_p = \begin{cases}
\{\gamma\in \Hom(D_p,C_4) : \gamma(I_p)=1\} & p=2,\\
\{\gamma\in \Hom(D_p,C_4) : \gamma(I_p)=C_4\} & \text{else.}
\end{cases}
\]
In this situation, $N_{\disc}(\Q,\Sigma;X)$ counts the number of tamely ramified $C_4$-extensions $L/\Q$ for which $|I_p(L/\Q)|\in\{1,4\}$ for all places $p$. The discriminant then satisfies $\nu_p(\disc(L/\Q)) = |I_p(L/\Q)|-1\in \{0,3\}$ as all places are at most tamely ramified, which implies $a_{\disc}(\Sigma)=3$. This is bigger than Malle's predicted invariant $a_{\disc}(C_4) = 1$. In fact, Wright's proof of Malle's conjecture for abelian extensions \cite{wright1989} extends to cases with restricted ramification in this way and shows that
\[
N_{\disc}(\Q,\Sigma;X) \sim c_{\disc}(\Q,\Sigma)X^{1/3}
\]
for some positive constant $c_{\disc}(\Q,\Sigma)$. This example highlights how the number of $G$-extensions with restricted local behaviors can be much smaller than the total number of $G$-extensions, and shows how the adjusted invariant $a_{\inv}(\Sigma)$ can capture this distinction.

\textbf{Remark:} By only considering $\gamma$ with $\gamma(I_\pK)\ne 1$ in the definition of $a_{\inv}(\Sigma)$, we guarantee that $0<a_{\inv}(\Sigma)\le \infty$ for all admissible orderings $\inv$ and local conditions $\Sigma$.

It is not necessarily true that we expect an analog of Malle's conjecture to hold for any choice of admissible ordering $\inv$ and local conditions $\Sigma$. As a counter example, we can choose $\Sigma$ to contain only local conditions which give Grunwald-Wang counter examples, in which case $N_{\inv}(K,\Sigma;X)=0$ (see \cite{wood2009} for a good exposition of this obstruction in the context of abelian extensions). We also defined admissibility for $\inv$ quite broadly compared to the existing literature, to the point where we allow strange orderings that make it hard to extend the theoretic justifications of Malle's conjecture.

We prove results towards the analog of Malle's predicted upper bound only, while keeping in mind that this bound may not be sharp in this generality:
\begin{conjecture}[Generalized Weak Malle's Conjecture]\label{conj:genWM}
Let $K$ be a number field, $G$ a finite group, $\inv$ an admissible ordering, and $\Sigma$ a family of local conditions. Then
\[
N_{\inv}(K,\Sigma;X) \ll X^{1/a_{\inv}(\Sigma) + \epsilon}\,,
\]
where we take the convention $1/a_{\inv}(\Sigma)=0$ if $a_{\inv}(\Sigma)=\infty$.
\end{conjecture}

All of the results stated in the introduction generalize to an arbitrary admissible ordering $\inv$ with any choice of restricted local conditions $\Sigma$. This produces very strong evidence that these bounds should be true in general, in particular the analog of Corollary \ref{cor:solvableintro} will imply that Conjecture \ref{conj:genWM} is true for all solvable groups if the Average Torsion Conjecture is true. We state the main result of the paper here, and the analogs of the other results stated in the introduction will be stated and proved in Section \ref{sec:cor}.

\begin{theorem}\label{thm:main}
Let $K$ be a number field, $G$ a finite group with abelian normal subgroup $N\normal G$, $\Sigma$ a family of local conditions and $\inv$ an admissible ordering. Then, for any $\epsilon>0$, $N_{\inv}(K,\Sigma;X)$ is
\[
\ll \max\left\{X^{1/a_{\inv}(\Sigma(N))+\epsilon}\ , \sum_{\substack{\pi\in \Surj^{(q_N)_*\Sigma}(G_K,G/N)\\ {\rm Nm}_{K/\Q}(\overline{\inv}(\pi))<X}} X^\epsilon\Big|\Hom\left(\Cl\left(L_\pi^{C_{G}(N)/N}\right),N\right)\Big|\right\}\,,
\]
where $L_\pi$ is the fixed field of $\ker\pi\le G_K$.
\end{theorem}

Theorem \ref{thm:main} is the full version of Theorem \ref{thm:mainintro} stated in the introduction.

\section{Technical Lemmas}\label{sec:technical}

We will need some intermediate results before jumping straight into the proof of Theorem \ref{thm:main}. We will use Dirichlet series as a convenient organizational tool for these counting functions, which we connect to counting functions of the form $N_{\inv}(K,\Sigma;X)$ via Integral transforms. It will suffice to use the following result:

\begin{lemma}\label{lem:Dirichlet}
Let $\{b_n\}$ be a sequence indexed by the positive integers for which $b_n \ge 0$ for all $n\ge 1$. For a fixed positive real number $a>0$ the following are equivalent:
\begin{enumerate}
\item[(i)]{
for all $\epsilon>0$
\[
\sum_{n\le X} b_n \ll X^{a+\epsilon}\,,
\]
}

\item[(ii)]{
The Dirichlet series $\ds \sum_{n=1}^{\infty} b_n n^{-s}$ converges for all real numbers $s>a$.
}
\end{enumerate}
\end{lemma}

We remark that it is only necessary to consider real values for $s$ (as all Dirichlet series converge on right half planes) and that we will \emph{not} utilize any analytic continuations. Not requiring the existence of an analytic continuation allows us to consider such a broad range of orderings and local conditions. The proof of some form of Lemma \ref{lem:Dirichlet} can be found in most introductory texts in analytic number theory (see \cite[Part II Theorem 1.13]{tenenbaum2015} for one such example).


We will also need upper bounds for counting abelian extensions which are uniform in the base field.

\begin{lemma}\label{lem:abelian}
Let $K$ be a number field, $A$ a finite abelian group, and $S\subset P_K$ a finite set of places. Then there exists a constant $C=C(A,m)$ depending only on the group $A$ and the degree $m=[K:\Q]$ such that
\[
|\Hom(G_K^S,A)| \le |\Hom(\Cl(K),A)| C(A,m)^{|S|}\,,
\]
where $G_K^S$ is the Galois group of the maximal extension of $K$ unramified outside of $S$.
\end{lemma}

\begin{proof}
Class Field Theory gives an exact sequence
\[
\begin{tikzcd}
\prod_{\pK\in S} I_\pK \rar & G_K^{S, {\rm ab}} \rar & \Cl(K) \rar & 1\,.
\end{tikzcd}
\]
The functor $\Hom(-,A)$ is a left exact functor of groups, which produces the following exact sequence of abelian groups:
\[
\begin{tikzcd}
\Hom(\Cl(K),A) \rar & \Hom(G_K^{S},A) \rar & \prod_{\pK\in S}\Hom(I_\pK, A)\,.
\end{tikzcd}
\]
$I_\pK$ is finitely generated, and in fact local class field theory tells us that $I_\pK \cong \text{cyclic}\times \Z_\ell^{[K:\Q]}$. This implies
\[
|\Hom(I_\pK,A)|\le |A|^{1+[K:\Q]}\,,
\]
so that
\[
|\Hom(G_K^S, G)|\le |\Hom(\Cl(K),G)|\cdot |A|^{(1+[K:\Q])|S|}\,.
\]
\end{proof}

\section{Proof of Theorem \ref{thm:main}}\label{sec:proofmain}

The main idea behind this section is the following: consider an exact sequence of finite groups
\[
\begin{tikzcd}
1 \rar & N \rar{i} & G \rar{q_N} & G/N \rar &1\,.
\end{tikzcd}
\]
If $G=A$ is an abelian group, then $\Hom(H,-) = H^1(H,-)$ is a left exact functor on abelian groups with the trivial $H$ action. In other words, the following is an exact sequence:
\[
\begin{tikzcd}
0 \rar & \Hom(H,N) \rar{i_*} & \Hom(H,A) \rar{(q_N)_*} &\Hom(H,A/N)\,.
\end{tikzcd}
\]
So in particular
\[
|\Hom(H,A)| \le |\Hom(H,N)|\cdot |\Hom(H,A/N)|\,.
\]
This was integral to the proof of Lemma \ref{lem:abelian} in the case $H=G_K^S$.

If $G$ is not abelian, the same statement is not true. The Hom functor only produces an exact sequence of pointed sets, which does not translate into bounds on the cardinality. Instead, we express $|\Hom(H,G)|$ as a sum of fibers of $(q_N)|_*$. We will require two technical results in group theory. First, we introduce some notation for this section:
\begin{itemize}
\item{
$G$ acts on $N$ by conjugation on the left, which we will denote $x.y=x y x^{-1}$ for $x\in G$ and $y\in N$. $N$ is a normal subgroup of $G$, so this is well-defined.
}
\item{
Define the homomorphism $\kappa:G\rightarrow \Aut(N)$ sending $x$ to the conjugation by $x$ map $(y\mapsto x y x^{-1})$ for any $y\in N$.
}
\item{
$\Aut(G)$ acts on $G$ by the natural left action, which we will often denote $\alpha.x=\alpha(x)$ for $x\in G$. This notation is used so that it lines up with conjugation $\kappa(x).y=x.y$.
}
\item{Given a group action $\phi:H\rightarrow \Aut(G)$ define the set of crossed homomorphisms
\[
Z^1_\phi(H,G) = \{f:H\rightarrow G \mid f(xy) = f(x) [\phi(x).f(y)]\}\,.
\]
}
\item{
Given any two maps $f,g:H\rightarrow G$ (not necessarily homomorphisms), define the map $(f*g):H\rightarrow G$ by coordinate-wise multiplication $x\mapsto f(x)g(x)$. Similarly for any set $B\subset \textnormal{Maps}(H,G)$ define $f*B=\{f*b | b\in B\}$ and $B*f=\{b*f|b\in B\}$. The operation $*$ makes $\textnormal{Maps}(H,G)$ into a group, but in general $\Hom(H,G)$ is not a subgroup because it is not closed.
}
\end{itemize}

\begin{lemma}\label{lem:homfiber}
Suppose $g\in \Hom(H,G)$, and $(q_N)_*(g)=\overline{g}$. Then the fiber above $\overline{g}$ is given by
\begin{align*}
q_*^{-1}(\overline{g}) &= Z^1_{\kappa g}(H,N) * g\,.
\end{align*}
\end{lemma}

\begin{proof}
Suppose there exists $g\in \Hom(H,G)$ such that $\beta_*(g)=\overline{g}$. For any $f\in Z^1_{\kappa g}(H,N)$, it follows that
\begin{align*}
(f*g)(xy) &= f(xy)g(xy)\\
&= f(x)[(\kappa g)(x).f(y)] g(x) g(y)\\
&=f(x)g(x)f(y) g(x)^{-1}g(x) g(y)\\
&= f(x)g(x)f(y) g(y)\\
&= (f*g)(x) (f*g)(y)\,.
\end{align*}
Clearly $q_*(f*g) = q_*(g)=\overline{g}$ by $\im f \subset N=\ker q$. Therefore $Z^1_{\kappa g}(H,N)*g\subset q_*^{-1}(\overline{g})$.

For the reverse containment, suppose $f\in \Hom(H,G)$ such that $q_*(f) = \overline{g}$. Then $f(x)g(x)^{-1}\in N$ for every $x\in H$ and
\begin{align*}
(f*g^{-1})(xy) &= f(xy)g(xy)^{-1}\\
&=f(x)f(y) g(y)^{-1} g(x)^{-1}\\
&= f(x) g(x)^{-1} g(x)(f(y)g(y)^{-1})g(x)^{-1}\\
&= (f*g^{-1})(x) [g(x).(f*g^{-1})(y)]\\
&= (f*g^{-1})(x) [(\kappa g)(x).(f*g^{-1})(y)]\,.
\end{align*}
Therefore $f*g^{-1}\in Z^1_{\kappa g}(H,N)$, so that $f\in Z^1_{\kappa g}(H,N) * g$. The opposite containment $q_*(\overline{g})\subset Z^1_{\kappa g}(H,N)*g$ then follows.
\end{proof}

This implies that we get an expression
\[
\Hom(H,G) = \bigcup_{\overline{g}\in {\rm im}((q_N)_*)} Z^1_{\kappa g}(H,N)
\]
for any choice of representative $g$ for each $\overline{g}$. Each fiber may be of a different size, and studying their asymptotic sizes when $H=G_K$ is a question of independent interest in number field counting (see \cite{alberts2019} for a more in depth discussion of these fibers in the context of number field counting). For our purposes, we only need an upper bound.

\begin{lemma}\label{lem:ztohom}
Let $H$ act on $G$ by the homomorphism $\phi:H\rightarrow \Aut(G)$. Then the restriction map $f\mapsto f|_{\ker\phi}$ defines a map
\begin{align*}
\alpha&: Z^1_\phi(H,G) \rightarrow \Hom(\ker\phi, G)\,,
\end{align*}
which has fibers of size at most $|H/\ker \phi|^{|G|}$.
\end{lemma}

\begin{proof}
Given any $f\in Z^1_{\phi}(H,G)$, the restriction $f|_{\ker\phi}$ belongs to $Z^1_{\phi}(\ker \phi,G)$. Note that $\phi|_{\ker \phi}:\ker \phi\rightarrow \Aut(G)$ is the trivial map, so $Z^1_{\phi}(\ker\phi, G)= \Hom(\ker \phi,G)$.

Suppose $f,g\in Z^1_{\phi}(H,G)$ such that $f|_{\ker\phi}=g|_{\ker\phi}$. Then $f*g^{-1}$ is a map sending $\ker \phi$ to $1$ such that
\begin{align*}
(f*g^{-1})(xy) &= f(xy)g(xy)^{-1}\\
&= f(x)[\phi(x).f(y)][\phi(x).g(y)]^{-1}g(x)^{-1}\\
&= f(x)g(x)^{-1} g(x)[\phi(x).(f(y)g(y)^{-1})]g(x)^{-1}\\
&= (f*g^{-1})(x) [(\kappa g)(x)\phi(x).(f*g^{-1})(y)]\,.
\end{align*}
For any $y\in H$ and $z\in \ker \phi$, $(f*g^{-1})(z)=1$ implies
\begin{align*}
(f*g^{-1})(yz) &= (f*g^{-1})(y) [(\kappa g)(y)\phi(y).(f*g^{-1})(z)]\\
&= (f*g^{-1})(y)\,.
\end{align*}
Therefore $(f*g^{-1})$ factors through the group of left cosets $H/\ker\phi$, so that the fiber $\alpha^{-1}(f|_{\ker\phi})$ embeds in $\textnormal{Maps}(H/\ker\phi,G)$, which has size $|H/\ker\phi|^{|G|}$.
\end{proof}

We will combine this with the result of Lemma \ref{lem:homfiber} in order to produce asymptotic upper bounds for $|\Hom(H,G)|$ when $H=G_K$.

\begin{proof}[Proof of Theorem \ref{thm:main}]
By Lemma \ref{lem:Dirichlet} it suffices to consider the Dirichlet series
\[
L_{\inv}(K,\Sigma;s)=\sum_{\pi\in \Surj^{\Sigma}(G_K,G)} {\rm Nm}_{K/\Q}(\inv(\pi))^{-s}\,.
\]
We decompose this series according to the fibers of $(q_N)_*$ as
\[
L_{\inv}(K,\Sigma;s)=\sum_{\overline{\pi}\in (q_N)_*(\Surj^{\Sigma}(G_K,G))} \sum_{\pi\in(q_N)_*^{-1}(\overline{\pi})}{\rm Nm}_{K/\Q}(\inv(\pi))^{-s}\,.
\]
Let $S(\pi)$ be the set of places ramified in $\pi$ which are not ramified in $\overline{\pi}$, and $S(\overline{\pi})$ the places ramified in $\overline{\pi}$. Lemma \ref{lem:homfiber} implies that
\[
\#\{\pi' \in (q_N)_*(\overline{\pi}) : S(\pi') \subset S(\pi)\} \le \Big|Z^1_{\kappa\pi}\left(G_K^{S(\pi)\cup S(\overline{\pi})},N\right)\Big|\,.
\]
The definitions of $\overline{\inv}$ and $a_{\inv}(\Sigma)$ imply that for any $a\le a_{\inv}(\Sigma)$
\begin{align*}
{\rm Nm}_{K/\Q}(\inv(\pi)) &= \prod_{\pK} {\rm Nm}_{K/\Q}(\pK)^{\inv_\pK(\pi)}\\
&= \prod_{\pK : \overline{\pi}(I_\pK)= 1}{\rm Nm}_{K/\Q}(\pK)^{\inv_\pK(\pi)}\cdot\prod_{\pK : \overline{\pi}(I_\pK)\ne 1}{\rm Nm}_{K/\Q}(\pK)^{\inv_\pK(\pi)}\\
&\gg\prod_{\pK\in S(\pi)}{\rm Nm}_{K/\Q}(\pK)^{a} \cdot \prod_{\pK : \overline{\pi}(I_\pK)\ne 1}{\rm Nm}_{K/\Q}(\pK)^{\overline{\inv}_\pK(\overline{\pi})}\\
&\gg\prod_{\pK\in S(\pi)}{\rm Nm}_{K/\Q}(\pK)^{a} \cdot {\rm Nm}_{K/\Q}(\overline{\inv}(\overline{\pi}))\,.
\end{align*}
(Note that it is necessary to look only at $a\le a_{\inv}(\Sigma)$ in order to include the case $a_{\inv}(\Sigma)=\infty$.) Putting these together gives the following upper bound for any $a\le a_{\inv}(\Sigma(N))$:
\begin{align*}
L_{\inv}(K,\Sigma;s) \ll&\sum_{\overline{\pi}\in (q_N)_*(\Surj^{\Sigma}(G_K,G))} \sum_{\substack{S\subset P_K\\\text{finite}\\S(\overline{\pi})\cap S=\emptyset}}\Big|Z^1_{\kappa\pi}\left(G_K^{S\cup S(\overline{\pi})},N\right)\Big|\\
&\times {\rm Nm}_{K/\Q}\left(\prod_{\pK\in S} \pK\right)^{-sa}{\rm Nm}_{K/\Q}(\overline{\inv}(\overline{\pi}))^{-s}\,,
\end{align*}
where $\pi$ is any choice of preimage of $\overline{\pi}$ that is ramified inside of $S$ (if such a choice does not exist, we treat that term of the sum as zero). However, the fact that $N$ is an abelian normal subgroup implies that it acts trivially on itself by conjugation. In particular, the action $\kappa\pi$ is determined uniquely by $\overline{\pi}$, and so is independent of the choice of lift. By taking advantage of this fact, we see that the summands no longer depend on the lift $\pi$.
\begin{align*}
L_{\inv}(K,\Sigma;s) \ll&\sum_{\overline{\pi}\in (q_N)_*(\Surj^{\Sigma}(G_K,G))} \sum_{\substack{S\subset P_K\\\text{finite}\\S(\overline{\pi})\cap S=\emptyset}}\Big|Z^1_{\kappa\overline{\pi}}\left(G_K^{S\cup S(\overline{\pi})},N\right)\Big|\\
&\times {\rm Nm}_{K/\Q}\left(\prod_{\pK\in S} \pK\right)^{-sa}{\rm Nm}_{K/\Q}(\overline{\inv}(\overline{\pi}))^{-s}\,.
\end{align*}
Lemma \ref{lem:ztohom} then implies that
\begin{align*}
\Big|Z^1_{\kappa\overline{\pi}}\left(G_K^{S\cup S(\overline{\pi})},N\right)\Big| \le |G_K^{S\cup S(\overline{\pi})}/\ker(\kappa\overline\pi)|^{|G|} |\Hom(\ker(\kappa\overline{\pi}),N)|\,.
\end{align*}
For every choice of $\overline{\pi}$, $|G_K^{S\cup S(\overline{\pi})}/\ker(\kappa\overline{\pi})|\le |\Aut(N)|$ is bounded independent of $S$. The Galois correspondence gives $\ker(\kappa\overline{\pi})=G_M^{S\cup S(\overline{\pi})}$ for $M=L^{\ker(\kappa\overline{\pi})}=L^{C_G(N)/N}$ where $L$ is the fixed field of $\ker \overline{\pi}$ and we let $G_M^{S\cup S(\overline{\pi})}$ denote the Galois group of the maximal extension unramified at all places $\mathcal{P}\in P_M$ such that $\mathcal{P}\mid \pK$ for some $\pK\in P_K-(S\cup S(\overline{\pi}))$. We note that $\overline{\pi}$ is surjective, so necessarily any field fixed by $\ker(\kappa\overline{\pi})$ is necessarily fixed by $\ker(\kappa)=C_G(N)$.

Lemma \ref{lem:abelian} implies that there exists a constant $C=C(N,m)$ independent of $S$ and $L$ such that
\begin{align*}
\Big|Z^1_{\kappa\overline{\pi}}\left(G_K^{S\cup S(\overline{\pi})},N\right)\Big| \ll& \Big|\Hom\left(\Cl\left(L^{C_{G}(N)/N}\right),N\right)\Big| C^{|S|+|S(\overline{\pi})|}\\
=& \Big|\Hom\left(\Cl\left(L^{C_{G}(N)/N}\right),N\right)\Big| C^{|S|+\omega(\overline{\inv}(\overline{\pi}))}\,,
\end{align*}
where $\omega(\mathfrak{a})=$ the number of distinct prime divisors of $\mathfrak{a}$. This implies
\begin{align*}
L_{\inv}(K,\Sigma;s) \ll&\sum_{\overline{\pi}\in (q_N)_*(\Surj^{\Sigma}(G_K,G))} \sum_{\substack{S\subset P_K\\\text{finite}\\S\cap S(\overline{\pi})=\emptyset}}\Big|\Hom\left(\Cl\left(L^{C_{G}(N)/N}\right),N\right)\Big|\\
&\times C^{|S|+\omega(\overline{\inv}(\overline{\pi}))}{\rm Nm}_{K/\Q}\left(\prod_{\pK\in S} \pK\right)^{-sa}{\rm Nm}_{K/\Q}(\overline{\inv}(\overline{\pi}))^{-s}\\
\ll&\sum_{\overline{\pi}\in (q_N)_*(\Surj^{\Sigma}(G_K,G))} \sum_{\substack{S\subset P_K\\\text{finite}}}\Big|\Hom\left(\Cl\left(L^{C_{G}(N)/N}\right),N\right)\Big|\\
&\times C^{|S|+\omega(\overline{\inv}(\overline{\pi}))}{\rm Nm}_{K/\Q}\left(\prod_{\pK\in S} \pK\right)^{-sa}{\rm Nm}_{K/\Q}(\overline{\inv}(\overline{\pi}))^{-s}\,.
\end{align*}
By removing the dependence of $S$ on $\overline{\pi}$ in the last step, we can factor the upper bound into a product of the following two Dirichlet series:
\begin{align*}
\sum_{\substack{S\subset P_K\\\text{finite}}}C^{|S|}{\rm Nm}_{K/\Q}\left(\prod_{\pK\in S} \pK\right)^{-sa}
\end{align*}
and
\begin{align*}
\sum_{\overline{\pi}\in (q_N)_*(\Surj^{\Sigma}(G_K,G))}\Big|\Hom\left(\Cl\left(L^{C_{G}(N)/N}\right),N\right)\Big|C^{\omega(\overline{\inv}(\overline{\pi}))}{\rm Nm}_{K/\Q}(\overline{\inv}(\overline{\pi}))^{-s}\,.
\end{align*}
The first Dirichlet series can be bounded as follows: noting that $C^{\omega(\mathfrak{a})}\ll {\rm Nm}_{K/\Q}(\mathfrak{a})^{\epsilon}$ for each $\epsilon>0$:
\begin{align*}
\sum_{\substack{S\subset P_K\\\text{finite}}}C^{|S|}{\rm Nm}_{K/\Q}\left(\prod_{\pK\in S} \pK\right)^{-sa}&\le\sum_{\mathfrak{a}\in I_K} C^{\omega(\mathfrak{a})}{\rm Nm}_{K/\Q}(\mathfrak{a})^{-sa}\\
&\ll \zeta_K(sa-\epsilon)\,.
\end{align*}
Recalling that $a$ can be chosen to be any number $\le a_{\inv}(\Sigma(N))$ and that the Dedekind zeta function $\zeta_K$ converges for all real numbers larger than $1$, it follows that the first series converges for all $s>1/a_{\inv}(\Sigma(N))$.

For the second Dirichlet series we directly apply Lemma \ref{lem:Dirichlet} to say that the series converges absolutely for all $s>b$ if and only if
\[
\sum_{\substack{\overline{\pi}\in (q_N)_*(\Surj^{\Sigma}(G_K,G))\\{\rm Nm}_{K/\Q}(\overline{\inv}(\overline{\pi}))<X}}\Big|\Hom\left(\Cl\left(L^{C_{G}(N)/N}\right),N\right)\Big|C^{\omega(\overline{\inv}(\overline{\pi}))} \ll X^{b+\epsilon}\,.
\]
Thus the series $L_{\inv}(K,\Sigma;s)$ converges absolutely for $s>\max\{1/a_{\inv}(\Sigma(N)), b\}$ for any positive real number $b$ satisfying the above inequality. Lemma \ref{lem:Dirichlet} then implies
\[
N_w(K,\Sigma;X) \ll \max\{X^{1/a_w(\Sigma(N))+\epsilon}, X^{b+\epsilon}\}\,.
\]
Taking a $\liminf$ of all the choices for $b$ concludes the proof of Theorem \ref{thm:main}
\end{proof}

\section{Proofs of the Corollaries}\label{sec:cor}

We first state and prove the analog of Corollary \ref{cor:solvableintro}. We will additionally be explicit about which averages of torsion in class groups need to be considered in order to prove Malle's predicted upper bounds.

\begin{corollary}\label{cor:solvable}
Let $K$ be a number field, $G$ a finite solvable group, $\inv$ an admissible ordering, and $\Sigma$ a family of local conditions. Suppose there exists a normal series
\[
1 \normal G_1 \normal G_2 \normal \cdots \normal G_{n} \normal G
\]
such that for each $i=1,2,...,n$ the series satisfies
\begin{enumerate}
\item[(a)]{
$G_i\normal G$,}

\item[(b)]{
the factor $G_i/G_{i-1}$ is abelian,
}

\item[(c)]{
the average size of $|\Hom(\Cl(L^{C_G(N)/N}),N)|$ for $N=G_i/G_{i-1}$ as $L/K$ varies over all $(G/G_i)$-extensions ordered by $\overline{\inv}$ satisfying the local conditions $(q_{G_i})_*\Sigma$ is $\ll X^{\epsilon}$, i.e.
\[
\ds \sum_{\substack{\pi\in \Surj^{(q_{G_i}^*)\Sigma}(G_K,G/G_i)\\ {\rm Nm}_{K/\Q}(\overline{\inv}(\pi))<X}} \Big|\Hom\left(\Cl\left(L^{C_{G}(G_i/G_{i-1})/G_i}\right),G_i/G_{i-1}\right)\Big|
\]
is $\ll X^{\epsilon}N_{\overline{\inv}}(K,(q_{G_i}^*)\Sigma;X)$, where $C_{G}(G_i/G_{i-1})$ is the centralizer of the action by conjugation of $G$ on $G_i/G_{i-1}$.
}
\end{enumerate}
Then
\[
N_{\inv}(K,\Sigma;X) \ll X^{1/a_{\inv}(\Sigma) + \epsilon}\,.
\]
\end{corollary}

This result gives strong evidence that Conjecture \ref{conj:genWM} holds for any solvable group under any admissible ordering with any restricted local conditions, as a normal series satisfying conditions (a) and (b) exists for any solvable group and (c) would follow from the average torsion conjecture for class groups. We immediately see that Corollaries \ref{cor:solvableintro} and \ref{cor:solvableell} follow from Corollary \ref{cor:solvable}.

\begin{proof}
We induct on the length of this normal series. Suppose $n=0$ as the base case, i.e. the series is just $1\normal G$ so that $G$ must be abelian. By Lemma \ref{lem:Dirichlet} it suffices to consider the Dirichlet series
\[
\sum_{\pi\in \Surj^{\Sigma}(G_K,G)} {\rm Nm}_{K/\Q}(\inv(\pi))^{-s}\,.
\]
By definition, for all but finitely many $\pK$ the exponent $\inv_\pK(\pi)\ge a_{\inv}(\Sigma)$. There are only finitely many choices for local behavior at the places violating this bound, which implies that for any real number $a\le a_{\inv}(\Sigma)$
\[
\sum_{\pi\in \Surj^{\Sigma}(G_K,G)} {\rm Nm}_{K/\Q}(\inv(\pi))^{-s}\ll \sum_{\pi\in \Surj^{\Sigma}(G_K,G)} {\rm Nm}_{K/\Q}\left(\prod_{\pK\mid \inv(\pi)}\pK\right)^{-sa}\,.
\]
(Note that it is necessary to look only at $a\le a_{\inv}(\Sigma)$ in order to include the case $a_{\inv}(\Sigma)=\infty$.)

Admissibility implies $\pK\mid \inv(\pi)$ if and only if $\pi(I_\pK)=1$ for all but finitely many places. Thus we get an upper bound
\[
\sum_{\pi\in \Surj^{\Sigma}(G_K,G)} {\rm Nm}_{K/\Q}\left(\prod_{\pK\mid \inv(\pi)}\pK\right)^{-sa} \le \sum_{\substack{S\subset P_K\\\text{finite}}} |\Hom(G_K^S,G)| {\rm Nm}_{K/\Q}\left(\prod_{\pK\in S} \pK\right)^{-sa}\,.
\]
We then apply Lemma \ref{lem:abelian} to get the bounds
\begin{align*}
&\sum_{\substack{S\subset P_K\\\text{finite}}} |\Hom(G_K^S,G)| \mathcal{N}_{K/\Q}\left(\prod_{\pK\in S} \pK\right)^{-sa}\\
&\hspace{1cm}\le \sum_{\substack{S\subset P_K\\\text{finite}}} |\Hom(\Cl(K),G)| C^{|S|}{\rm Nm}_{K/\Q}\left(\prod_{\pK\in S} \pK\right)^{-sa}\\
&\hspace{1cm}\le \sum_{\mathfrak{a}\in I_K} |\Hom(\Cl(K),G)| C^{\omega(\mathfrak{a})}{\rm Nm}_{K/\Q}(\mathfrak{a})^{-sa}\,,
\end{align*}
where $\omega(\mathfrak{a})$ is the number of distinct prime divisors of $\mathfrak{a}$. Noting that $C^{\omega(\mathfrak{a})}\ll {\rm Nm}_{K/\Q}(\mathfrak{a})^{\epsilon}$ for each $\epsilon>0$, we get an upper bound given by the Dedekind zeta function
\[
|\Hom(\Cl(K),G)| \zeta_K\left(sa-\epsilon\right)\,,
\]
which is known to converge absolutely for $sa-\epsilon>1$. As $a$ can be chosen to be any number $\le a_{\inv}(\Sigma)$, it follows that the series converges for all $s>1/a_{\inv}(\Sigma)$. Lemma \ref{lem:Dirichlet} then converts this result to the desired bound
\[
N_{\inv}(K,\Sigma;X) = \sum_{\substack{\pi\in \Surj^{\Sigma}(G_K,G)\\ {\rm Nm}_{K/\Q}(\inv(\pi))\le X}} 1= O(X^{1/a_{\inv}(\Sigma) + \epsilon})\,.
\]

Now suppose the theorem is true for groups that have a normal series of length $n-1$. We then apply Theorem \ref{thm:main} for $N=G_1$ the abelian normal subgroup of $G$ to show $N_{\inv}(K,\Sigma;X)$ is
\begin{align*}
\ll \max\left\{X^{1/a_{\inv}(\Sigma(N))+\epsilon}\ , \sum_{\substack{\pi\in \Surj^{(q_n)_*\Sigma}(G_K,G/N)\\ {\rm Nm}_{K/\Q}(\overline{\inv}(\pi))<X}} X^{\epsilon}\Big|\Hom\left(\Cl\left(L^{C_{G}(N)/N}\right),N\right)\Big|\right\}\,.
\end{align*}
We assumed a bound of the average size of class group torsion in this case, namely that for any $\epsilon>0$
\[
\sum_{\substack{\pi\in \Surj^{(q_n)_*\Sigma}(G_K,G/N)\\ {\rm Nm}_{K/\Q}(\overline{\inv}(\pi))<X}} \Big|\Hom\left(\Cl\left(L^{C_{G}(N)/N}\right),N\right)\Big| \ll X^{\epsilon} N_{\overline{\inv}}(K,(q_N)_*\Sigma;X)\,.
\]
It follows from the inductive hypothesis that
\[
N_{\inv}(K,\Sigma;G) \ll \max\left\{X^{1/a_{\inv}(\Sigma(N))+\epsilon}\ , X^{1/a_{\overline{\inv}}((q_N)_*\Sigma)+\epsilon}\right\}\,.
\]
We now note that
\begin{align*}
a_{\inv}(\Sigma) &= \liminf_{{\rm Nm}_{K/\Q}(\pK)\rightarrow \infty} \min_{\substack{\gamma\in \Sigma_\pK\\ \gamma(I_\pK)\ne 1}} \inv_\pK(\gamma)\\
&=\liminf_{{\rm Nm}_{K/\Q}(\pK)\rightarrow \infty} \min\left\{\min_{\substack{\gamma\in \Sigma_\pK(N)\\ \gamma(I_\pK)\ne 1}} \inv_\pK(\gamma),\min_{\substack{\gamma\in \Sigma_\pK-\Sigma_{\pK}(N)\\ \gamma(I_\pK)\ne 1}} \inv_\pK(\gamma) \right\}\\
&=\liminf_{{\rm Nm}_{K/\Q}(\pK)\rightarrow \infty} \min\left\{\min_{\substack{\gamma\in \Sigma_\pK(N)\\ \gamma(I_\pK)\ne 1}} \inv_\pK(\gamma),\min_{\substack{\gamma\in (q_N)_*\Sigma_\pK\\ \gamma(I_\pK)\ne 1}} \overline{\inv}_\pK(\gamma) \right\}\\
&=\min\{a_{\inv}(\Sigma), a_{\overline{\inv}}((q_N)_*\Sigma)\}\,,
\end{align*}
which concludes the proof.
\end{proof}

We next state the analog of the unconditional result Corollary \ref{cor:hypercenter}:

\begin{corollary}
Let $K$ be a number field, $G$ a finite group, $\inv$ an admissible ordering, and $\Sigma$ a family of local conditions. Suppose there exists a normal subgroup $N\normal G$ such that
\begin{enumerate}
\item[(a)]{
$N/(N\cap Z^{\infty}(G))$ is abelian with $d$ generators, where $Z^{\infty}(G)$ denotes the hypercenter of $G$ i.e. the last term of the lower central series,
}

\item[(b)]{
There exists a positive real number $M$ such that for each $G/N$-extension $L/K$ ${\rm Nm}_{K/\Q}(\disc(L/K))\ll {\rm Nm}_{K/\Q}(\overline{\inv}(L/K))^M$ and
\[
X^{\frac{1}{2} d M} \cdot N_{\overline{\inv}}(K,(q_N)_*\Sigma;X) \ll X^{1/a_{\inv}(\Sigma)+\epsilon}\,.
\]
}
\end{enumerate}
Then
\[
N_{\inv}(K,\Sigma;X) \ll X^{1/a_{\inv}(\Sigma) + \epsilon}\,.
\]
\end{corollary}

We immediately see how this is a generalization of Corollary \ref{cor:hypercenter}, and we discussed in the introduction how taking $N=G$ for $G$ nilpotent implies Corollary \ref{cor:nilpotent} (as well as the appropriate analog for any admissible ordering and any family of local conditions).

\begin{proof}
We choose a composition series
\[
1\normal N_1 \normal N_2 \normal N_3 \normal \cdots \normal N_n \normal N
\]
by specifying $N_1,...,N_n$ as the lower central series for $G$ intersected with $N$, so that in particular $N_n= (N\cap Z^{\infty}(G))$. This implies that for each $i=1,...,n$, $N_i$ is central in $G/N_{i-1}$ so that $C_G(N_i/N_{i-1})/N_i = G$. Thus for each $i=1,...,n$ it follows that
\begin{align*}
&\sum_{\substack{\pi\in \Surj^{(q_{N_i})_*\Sigma}(G_K,G/N_i)\\ {\rm Nm}_{K/\Q}(\overline{\inv}(\pi))<X}} X^\epsilon\Big|\Hom\left(\Cl\left(L_\pi^{C_{G}(N_i/N_{i-1})/N_i}\right),N_i/N_{i-1}\right)\Big|\\
&\hspace{1cm}\ll |\Hom(\Cl(K),N_i/N_{i-1})| X^\epsilon N_{\overline{\inv}}(K,(q_{N_i})_*\Sigma;X)\,.
\end{align*}
We also note that Minkowski's bound for the class group implies
\begin{align*}
&\sum_{\substack{\pi\in \Surj^{(q_{N})_*\Sigma}(G_K,G/N)\\ {\rm Nm}_{K/\Q}(\overline{\inv}(\pi))<X}} X^\epsilon\Big|\Hom\left(\Cl\left(L_\pi^{C_{G}(N/N_{n})/N}\right),N/N_{n}\right)\Big|\\
&\hspace{1cm}\ll\sum_{\substack{\pi\in \Surj^{(q_{N})_*\Sigma}(G_K,G/N)\\ {\rm Nm}_{K/\Q}(\overline{\inv}(\pi))<X}} X^\epsilon{\rm Nm}_{K/\Q}(\disc(L/K))^{\frac{1}{2}d}\\
&\hspace{1cm}\ll\sum_{\substack{\pi\in \Surj^{(q_{N})_*\Sigma}(G_K,G/N)\\ {\rm Nm}_{K/\Q}(\overline{\inv}(\pi))<X}} X^\epsilon{\rm Nm}_{K/\Q}(\overline{\inv}(L/K))^{\frac{1}{2}dM}\\
&\hspace{1cm}\ll X^{\frac{1}{2}dM} N_{\overline{\inv}}(K,(q_N)_*\Sigma;X)\\
&\hspace{1cm}\ll X^{1/a_{\inv}(\Sigma)+\epsilon}\,.
\end{align*}
Iterating along the composition series with Theorem \ref{thm:main} then implies $N_{\inv}(K,\Sigma;X)$ is
\begin{align*}
&\ll \max\left\{X^{1/a_{\inv}(\Sigma(N_1))+\epsilon}, ..., X^{1/a_{\overline{\inv}}((q_{N_{n}})_*\Sigma(N_n))+\epsilon},X^{1/a_{\inv}(\Sigma)+\epsilon}\right\}\\
&\ll X^{1/a_{\inv}(\Sigma)+\epsilon}\,.
\end{align*}
\end{proof}

\section{Remarks on Possible Generalizations}\label{sec:remarks}

Theorem \ref{thm:main} is not limited to solvable groups and can be used to address nonsolvable groups as well. The importance of solvable comes from class field theory - we know a lot more about abelian extensions than we do for $G$-extensions for $G$ a nonabelian simple group. Theorem \ref{thm:main} only allows for taking an abelian normal subgroup $N\normal G$ because we understand the field theory for $N$-extension very well. If we chose another normal subgroup $N\normal G$ the proof would break down when trying to analyze $|\Hom(G_K^S,N)|$. We don't have a nonabelian version of class field theory to give the upper bounds analogous to Lemma \ref{lem:abelian} necessary to prove the theorem.

As in the case for solvable groups, if one could show that there exists a positive constants $C=C(G,m)$ and $c=c(K,N)$ independent of $S$ such that $|\Hom(G_K^{S},N)|\ll  c(K,N)\cdot C(N,m)^{|S|}$ for any nonabelian simple group $N$ and we assume the number of unramified $N$-extensions of $L/K$ is smaller than $X^{\epsilon}$ on average, we could prove Malle's predicted upper bounds. Although these assertions match the corresponding result for abelian groups and the average torsion conjecture for class groups, there is significantly less evidence for bounds of these forms in the literature. Even in the case of solvable groups, one could make improvements by generalizing the bounds in Lemma \ref{lem:abelian} to other nonabelain solvable groups. This would in principle allow one to take the normal subgroup $N$ in Theorem \ref{thm:main} to be some nonabelian group, which could improve upper bounds by decreasing the size of $G/N$. This approach is reminiscent of the work of Wang \cite{jwang2017} on Malle's conjecture for $S_n\times A$, in which she proved upper bounds for $S_n$-extensions for $n=3,4,5$ which are uniform in the base field.

We were able to prove very broad results due to the fact that we restricted to considering only the upper bound as a power of $X$. These results can be traced through to keep track of the power of $\log X$, although more care is needed to get the sharpest results possible. In the cases of Corollary \ref{cor:hypercenter} for which Malle's predicted upper bound as a power of $X$ is known, it may be the case that with more care the techniques of \cite{kluners2012,jwang2017}, and an upcoming preprint \cite{lemke-oliver-jwang-wood2019} could be used to prove an upper bound with the correct power of $\log X$ as well. Optimistically, one would hope to generalize these approaches in order to prove the strong form of Malle's conjecture in as wide a range of cases as Corollary \ref{cor:hypercenter} if we are more restrictive on the ordering and families of local conditions. See \cite{alberts2019} for the first steps in this direction.

\appendix

\section{Appendix: Data}\label{app:data}

In this section, we will provide data for the upper bounds of $N(K,G;X)$ when $G\subset S_n$ is a transitive, solvable subgroup and $n$ small. We will directly compare the bounds given by iteratively applying Theorem \ref{thm:main} with Minkowski's bounds on the size of the class group to the best previously known bounds, in particular the general bounds for every group given by Dummit \cite{dummit2018} and Schmidt \cite{schmidt1995}.

Some families of groups are easy to produce bounds for $N(K,G;X)$ using computations done by hand, such as $D_n\subset S_n$ as discussed in the introduction. In general though, we can get a more complete picture by using a computer algebra program. All of the computations in this section are done using MAGMA \cite{MAGMA}.

One of the drawbacks of Dummit's result is the computational power necessary to compute sets of primary invariants, Dummit's data extends to transitive groups of degree 8 and then covers only four transitive groups of degree 9 because of the length of time computations were taking. If we apply the trivial bound from Minkowski to Theorem \ref{thm:main} the bulk of the computations are done by computing a normal series for $G$ and looping over elements of $G$ to compute $a(G)$ and the new upper bounds, which MAGMA is able perform very quickly by comparison.

We briefly describe the code being used in this section: For each transitive group $G$ in degree $d$, we iterate through the lattice of normal subgroups of $G$ to produce a list of all chief series for the group (i.e., normal series of maximal length). For each chief series
\[
1 \normal G_1 \normal G_2 \normal \cdots \normal G_n \normal G\,,
\]
we induct along the series by applying Theorem \ref{thm:main} at the bottom of the series as follows: If $G_1$ is central, apply Theorem \ref{thm:main} with $N=G_1$. If $G_1$ is not central, apply Theorem \ref{thm:main} with $N=G_i$ for $i$ the maximum index for which $G_i$ is abelian. In the first case, $N$ is central so the class group contributes nothing to the upper bound. In the second case, we use Minkowski's bound in order to give an upper bound for the contribution of the class group. We have left over a chief series for $G/N$ which is strictly shorter, and we have MAGMA repeat this process until we reach the end of the chief series. This is done for \emph{each} chief series, as different series give different bounds, and the program returns the minimum bound produced by one of the chief series.

We will use nTd to denote the group TranstiveGroup(n,d) in MAGMA's database, and we will only include solvable groups. In each column, we will give the corresponding power of $X$ upper bound of $N(K,G;X)$: the ``Malle" column which shows the upper bound predicted by Malle, the ``new" column which shows the unconditional bounds we prove using Theorem \ref{thm:main} with Minkowski's bounds, the ``previous" column which shows the previously best known result (or ``SF" if the strong form of Malle's conjecture is proven), and the ``reference" column which gives the reference for the previously best known result. We remark that Dummit's bounds depend slightly on the field $K$, if $X^{a}$ is the bound given over $\Q$ then $X^{a + 1 - 1/[K:\Q]}$ is the corresponding bound over $K$. Whenever Dummit's bounds are the best known, we specifically include the bound over $\Q$.

To preserve space, we include only those groups for which the new upper bounds produced are better than the previously known bounds. In two case, namely 8T33 and 8T34, the new bounds improve on Dummit's upper bounds whenever $K\ne \Q$, so we include these with an *. We remark that for groups where the new upper bounds are not better than the previously known bounds, it is possible for the new bounds to be extremely poor. The worst example of small degree is
\[
N(K,\text{8T43};X) \ll X^{77+\epsilon}\,,
\]
which is significantly worse than Dummit's bound of $X^{7/3}$. This bound is so large because 8T43 has a chief series of length 7, of which most of the factors are not central. This means that the class group increases the size of the bound at more steps of the induction.

The bounds we produce also tie with the best previously known bounds in several cases, including all nilpotent groups in the regular representation \cite{kluners-malle2004}, all $\ell$-groups \cite{kluners-malle2004}, and dihedral groups $D_p$ in both representations, all of which will be omitted from the table. Mehta \cite{mehta2019} announced some results around the same time as this paper: he proves Malle's predicted upper bound for 6T4 and 6T10, and although our results are better than the previously known upper bounds before Mehta's result we do not produce Malle's predicted upper bound. Thus we exclude 6T4 and 6T10 from the tables. Mehta also proves unconditional upper bounds for $C_m\rtimes C_t\subset S_n$ for $n=m,mt$ which agree with the bounds we produce. This covers the groups 7T3, 7T4, 9T3, and 9T10 in our table, and we choose to include these bounds with a $\dagger$ to indicate the concurrent result.


{
\begin{center}
\begin{longtable}{|c|c|c|c|c|c|}
\hline
degree 6 & Isom. to & Malle & new & previous & reference\\
\hline
6T3 & $S_3\times C_2$ & $X^{1/2+\epsilon}$  & $X^{3/4+\epsilon}$ & $X^{7/3}$ & Dummit \cite{dummit2018}\\
6T5 & $C_3\wr C_2$ & $X^{1/2+\epsilon}$ & $X^{1/2+\epsilon}$ & $X^{7/4}$ & Dummit \cite{dummit2018}\\
6T8 & $S_4$ & $X^{1/2+\epsilon}$ & $X^{3/2+\epsilon}$  & $X^2$ & Schmidt \cite{schmidt1995}\\
6T9 & $S_3\times S_3$ & $X^{1/2+\epsilon}$ & $X^{3/2+\epsilon}$ & $X^2$ & Schmidt \cite{schmidt1995}\\
\hline
%
%
\hline
degree 7 & Isom. to & Malle & new & previous & reference\\
\hline
7T3$^\dagger$ & $F_{21}$ & $X^{1/4+\epsilon}$ & $X^{1/2+\epsilon}$ & $X^{7/4}$ & Dummit \cite{dummit2018}\\
7T4$^\dagger$ & $F_{42}$ & $X^{1/3+\epsilon}$ & $X^{5/6+\epsilon}$ & $X^2$ & Dummit \cite{dummit2018}\\
\hline
%
%
\hline
degree 8 & Isom. to & Malle & new & previous & reference\\
\hline
8T12 & $SL_2(\F_3)$ & $X^{1/4+\epsilon}$ & $X^{3/4+\epsilon}$ & $X^{5/2}$ & Schmidt \cite{schmidt1995}\\
8T13 & $A_4\times C_2$ & $X^{1/4+\epsilon}$ & $X^{3/4+\epsilon}$ & $X^{5/2}$ & Schmidt \cite{schmidt1995}\\
8T14 & $S_4$ & $X^{1/4+\epsilon}$ & $X^{11/8+\epsilon}$ & $X^{5/2}$ & Schmidt \cite{schmidt1995}\\
8T23 & $GL_2(\F_3)$ & $X^{1/3+\epsilon}$ & $X^{3/2+\epsilon}$ & $X^{5/2}$ & Schmidt \cite{schmidt1995}\\
8T24 & $S_4\times C_2$ & $X^{1/2+\epsilon}$ & $X^{9/4+\epsilon}$ & $X^{5/2}$ & Schmidt \cite{schmidt1995}\\
8T32 & & $X^{1/2+\epsilon}$ & $X^{5/4+\epsilon}$ & $X^{5/2}$ & Schmidt \cite{schmidt1995}\\
8T33$^*$ & $C_2^2\rtimes C_6$ & $X^{1/2+\epsilon}$ & $X^{9/4+\epsilon}$ & $X^2$ & Dummit \cite{dummit2018}\\
8T34$^*$ & $E_4^2 \rtimes D_6$ & $X^{1/2+\epsilon}$ & $X^{19/8+\epsilon}$ & $X^2$ & Dummit \cite{dummit2018}\\
\hline
%
%
%
\hline
degree 9 & Isom. to & Malle & new & previous & reference\\
\hline
9T3$^\dagger$ & $D_9$ & $X^{1/4+\epsilon}$ & $X^{3/8+\epsilon}$ & $X^{13/6}$ & Dummit \cite{dummit2018}\\
9T5 & $C_3^2 \rtimes C_2$ & $X^{1/4+\epsilon}$ & $X^{1/2+\epsilon}$ & $X^{19/12}$ & Dummit \cite{dummit2018}\\
9T8 & $S_3\times S_3$ & $X^{1/3+\epsilon}$ & $X^{1+\epsilon}$ & $X^2$ & Dummit \cite{dummit2018}\\
9T9 & & $X^{1/4+\epsilon}$ & $X^{3/4+\epsilon}$ & $X^{11/4}$ & Schmidt \cite{schmidt1995}\\
9T10$^\dagger$ & & $X^{1/4+\epsilon}$ & $X^{3/4+\epsilon}$ & $X^{11/4}$ & Schmidt \cite{schmidt1995}\\
9T11 & & $X^{1/4+\epsilon}$ & $X^{9/8+\epsilon}$ & $X^{11/4}$ & Schmidt \cite{schmidt1995}\\
9T12 & & $X^{1/3+\epsilon}$ & $X^{2/3+\epsilon}$ & $X^{11/4}$ & Schmidt \cite{schmidt1995}\\
9T13 & & $X^{1/3+\epsilon}$ & $X^{4/3+\epsilon}$ & $X^{11/4}$ & Schmidt \cite{schmidt1995}\\
9T14 & & $X^{1/4+\epsilon}$ & $X^{5/4+\epsilon}$ & $X^{11/4}$ & Schmidt \cite{schmidt1995}\\
9T15 & & $X^{1/4+\epsilon}$ & $X^{5/4+\epsilon}$ & $X^{11/4}$ & Schmidt \cite{schmidt1995}\\
9T16 & & $X^{1/3+\epsilon}$ & $X^{5/3+\epsilon}$ & $X^{11/4}$ & Schmidt \cite{schmidt1995}\\
9T18 & & $X^{1/3+\epsilon}$ & $X^{5/2+\epsilon}$ & $X^{11/4}$ & Schmidt \cite{schmidt1995}\\
9T20 & & $X^{1/2+\epsilon}$ & $X^{3/2+\epsilon}$ & $X^{11/4}$ & Schmidt \cite{schmidt1995}\\
9T21 & & $X^{1/2+\epsilon}$ & $X^{3/2+\epsilon}$ & $X^{11/4}$ & Schmidt \cite{schmidt1995}\\
9T22 & & $X^{1/2+\epsilon}$ & $X^{11/6+\epsilon}$ & $X^{11/4}$ & Schmidt \cite{schmidt1995}\\
9T24 & & $X^{1/2+\epsilon}$ & $X^{7/2+\epsilon}$ & $X^{11/4}$ & Schmidt \cite{schmidt1995}\\
\hline
\end{longtable}
\end{center}
}
\vspace{-0.75cm}
We remark that Dummit only computes the bounds for groups 9T3, 9T4, 9T5, and 9T8 in \cite{dummit2018}. Dummit's Theorem does give bounds for all proper transitive subgroups $G\subset S_n$ which are known to be better that Schmidt's bounds if $G$ is primitive, but it becomes computationally intensive to find a set of primary invariants in order to compute the bound (it took Dummit's code two days to produce the bounds in degrees 5, 6, 7, 8, and for just these four groups in degree 9).

One should notice that our new bounds appear to improve many more results in degree $9$ than in degree $8$. This is for several reasons. $|S_8|$ is divisible by a much much larger power of $2$ than the power of $3$ dividing $|S_9|$, which means there are a lot more $2$-groups in $S_8$ for which (conjecturally) sharp bounds were already proven by Kl\"uners-Malle. Because $8$ is divisible by $2$, $S_8$ also has some transitive subgroups of the form $C_2\wr H$ for which the strong form of Malle's conjecture holds \cite{kluners2012}, while $S_9$ has no such subgroups. Lastly, $8$ has more divisors than $9$, which means there are a lot more ways to make transitive subgroups with longer normal series and we know that that the unconditional bounds coming from Theorem \ref{thm:main} with Minkowski's bounds get worse for longer normal series.

We would expect this pattern to continue to hold for larger degrees. Our new results are likely to improve the best known upper bounds for more groups in odd degrees with fewer divisors.

\section*{Acknowledgements}

I would like to thank my advisor Nigel Boston, Jordan Ellenberg, J\"urgen Kl\"uners, and Melanie Matchett Wood for many helpful conversations and recommendations. I would also like to thank anonymous referees for providing additional feedback during the submission process. This work was done with the support of National Science Foundation grant DMS-1502553.

\bibliographystyle{alpha}

\bibliography{../../BAreferencesV2019}

Department of Mathematics, University of Connecticut, 341 Mansfield Rd, Storrs, CT 06269 USA\\
\emph{E-mail address: brandon.alberts@uconn.edu}
\end{document}